\documentclass[oneside,reqno,american]{amsart}
\usepackage[T1]{fontenc}
\usepackage[utf8]{inputenc}
\usepackage{xcolor}
\usepackage{babel}
\usepackage{prettyref}
\usepackage{mathtools}
\usepackage{amstext}
\usepackage{amsthm}
\usepackage{amssymb}
\usepackage[pdfusetitle,
 bookmarks=true,bookmarksnumbered=false,bookmarksopen=false,
 breaklinks=false,pdfborder={0 0 0},pdfborderstyle={},backref=false,colorlinks=false]
 {hyperref}
\hypersetup{
 colorlinks=true,citecolor=blue,linkcolor=blue,linktocpage=true}

\makeatletter
\numberwithin{equation}{section}
\numberwithin{figure}{section}

\newrefformat{cor}{Corollary~\ref{#1}}
\newrefformat{subsec}{Section~\ref{#1}}
\newrefformat{lem}{Lemma~\ref{#1}}
\newrefformat{thm}{Theorem~\ref{#1}}
\newrefformat{sec}{Section~\ref{#1}}
\newrefformat{chap}{Chapter~\ref{#1}}
\newrefformat{prop}{Proposition~\ref{#1}}
\newrefformat{exa}{Example~\ref{#1}}
\newrefformat{tab}{Table~\ref{#1}}
\newrefformat{rem}{Remark~\ref{#1}}
\newrefformat{def}{Definition~\ref{#1}}
\newrefformat{fig}{Figure~\ref{#1}}
\newrefformat{claim}{Claim~\ref{#1}}

\AtBeginDocument{
  
}

\makeatother

\theoremstyle{plain}
\newtheorem{thm}{\protect\theoremname}
\newtheorem{lem}[thm]{\protect\lemmaname}
\newtheorem{prop}[thm]{\protect\propositionname}
\newtheorem{cor}[thm]{\protect\corollaryname}
\theoremstyle{remark}
\newtheorem{rem}[thm]{\protect\remarkname}
\newrefformat{cor}{Corollary \ref{#1}}
\newrefformat{lem}{Lemma \ref{#1}}
\newrefformat{prop}{Proposition \ref{#1}}
\newrefformat{rem}{Remark \ref{#1}}
\newrefformat{sec}{Section~\ref{#1}}
\newrefformat{subsec}{Section~\ref{#1}}
\newrefformat{thm}{Theorem \ref{#1}}
\providecommand{\corollaryname}{Corollary}
\providecommand{\lemmaname}{Lemma}
\providecommand{\propositionname}{Proposition}
\providecommand{\remarkname}{Remark}
\providecommand{\theoremname}{Theorem}

\begin{document}
\subjclass[2020]{Primary 41A30. Secondary 41A25,  46E22, 60E15, 68T07.}
\title{Ridge Kernel Averaging and Uniform Approximation}
\begin{abstract}
We develop a framework for function classes generated by parametric
ridge kernels: one-dimensional kernels composed with affine projections
and averaged over a parameter measure. The induced kernels are positive
definite, and the resulting integral class coincides isometrically
with its reproducing kernel Hilbert space. We characterize all kernels
obtainable by varying the measure as the uniform closure of the conic
hull of ridge atoms, giving a sharp universality criterion; a slice-wise
polynomial versus non-polynomial dichotomy governs expressivity. We
then analyze random-kernel networks whose activations may be indefinite
but have a positive-definite mean. For these networks we prove a Monte
Carlo rate with mean-squared error of order one over $N$ and a high-probability
uniform bound on compact sets, without requiring pathwise positive
definiteness.
\end{abstract}

\author{James Tian}
\address{Mathematical Reviews, 535 W. William St, Suite 210, Ann Arbor, MI
48103, USA}
\email{jft@ams.org}
\keywords{reproducing kernel Hilbert space (RKHS); positive definite kernels;
ridge functions; kernel synthesis; universality; Barron spaces; random
features; random-kernel networks; Monte Carlo approximation.}
\maketitle

\section{Introduction}\label{sec:1}

The connection between wide neural networks and kernel methods has
emerged as a foundational insight in modern machine learning theory.
As the width of a layer grows large (or infinite), a network’s behavior
often converges to that of a kernel machine, with the kernel determined
by the architecture and initialization. This viewpoint opens a rigorous
pathway to understanding networks through the well-developed theory
of reproducing kernel Hilbert spaces (RKHSs) \cite{MR51437,MR2450103}.

A central thread in this development is the \textit{random feature}
(RF) model of Rahimi and Recht \cite{10.5555/2981562.2981710}. One
starts from a fixed p.d. kernel $K$ with a feature decomposition
$K(x,x')=\mathbb{E}_{\omega}[\phi_{\omega}(x)\phi_{\omega}(x')]$,
then approximates $K$ by Monte Carlo sampling of $\phi_{\omega}$.
The network output is a linear combination of these sampled features,
so training reduces to a convex optimization over the output weights.
RF methods yield efficient kernel approximation and sharp insight
into the expressive power of shallow networks \cite{10.5555/2981780.2981944,10.5555/3020847.3020936}.
In most classical constructions (e.g., Fourier features for Gaussian
kernels \cite{10.5555/2981562.2981710}; ReLU features in \cite{10.5555/2984093.2984132}),
the random feature map is \textit{pathwise p.d.}: each realization
$\ensuremath{\omega\mapsto\phi_{\omega}\left(x\right)\phi_{\omega}\left(x'\right)}$
is itself a p.d. kernel.

\textbf{Our perspective.} Pathwise positive definiteness is not necessary.
It suffices that a \textit{mean kernel} in one dimension be p.d. and
continuous,
\[
K\left(s,t\right)=\mathbb{E}_{\omega}\left[k\left(\omega,s,t\right)\right],
\]
even if the individual random kernels $k\left(\omega,\cdot,\cdot\right)$
are signed or indefinite. This shift, from enforcing structure on
each sample to enforcing it only in expectation, leads to a broader
class of \textit{random-kernel networks}, in which each neuron computes
\[
x\mapsto k\left(\omega,\left\langle a,x\right\rangle +b,t\right).
\]

We analyze the deterministic function classes that arise from the
mean kernel $\ensuremath{K}$ via \textit{parametric ridge kernels}
\[
\psi_{\left(a,b,t\right)}\left(x\right)=K\left(\left\langle a,x\right\rangle +b,t\right),\qquad\left(a,b,t\right)\in\mathbb{R}^{d}\times\mathbb{R}\times\mathbb{R}\eqqcolon Z,
\]
and the induced kernel and integral class obtained by averaging these
atoms against a finite Borel measure $\rho$ on the parameter space:
\begin{align*}
K_{\rho}\left(x,x'\right) & =\int_{Z}\psi_{z}\left(x\right)\psi_{z}\left(x'\right)d\rho\left(z\right),\\
H_{\rho} & =\left\{ x\mapsto{\textstyle \int_{Z}c\left(z\right)\psi_{z}\left(x\right)d\rho\left(z\right):c\in L^{2}\left(\rho\right)}\right\} .
\end{align*}

We then show that finite networks built from the random kernels $k$
provide Monte Carlo approximations to functions in $H_{\rho}$ without
ever approximating a kernel directly. The resulting guarantees rely
only on the mean-p.d. property of $K$, not on pathwise p.d. of $k$.

The main results are as follows:

\prettyref{sec:2} deals with RKHS identification and kernel synthesis.
For any finite Borel measure $\rho$, the induced kernel $K_{\rho}$
is continuous and p.d. on $X\subset\mathbb{R}^{d}$, and $H_{\rho}$
is (isometrically) the RKHS $H_{K_{\rho}}$ (\prettyref{lem:1}).
Writing $\mathcal{F}=\overline{\{\psi_{z}:z\in Z\}}^{\|\cdot\|_{\infty}}\subset C(X)$,
we prove that the uniform closure of $\left\{ K_{\rho}\right\} $
coincides with the uniform closure of the conic hull $\left\{ \sum_{j}w_{j}f_{j}\otimes f_{j}:w_{j}\ge0,\ f_{j}\in\mathcal{F}\right\} $
(\prettyref{prop:3}). Hence our construction produces all continuous
p.d. kernels on $X$ if and only if $\mathcal{F}=C\left(X\right)$
(\prettyref{cor:4}), separating the familiar span-dense universal
approximation of ridge dictionaries from the stronger set-dense requirement
needed for full kernel synthesis.

We obtain a universality dichotomy via one-dimensional slices. Expressivity
hinges on the slices $s\mapsto K\left(s,t\right)$. If every slice
is a polynomial of degree $\le m$, then each $H_{\rho}$ lies in
the space of multivariate polynomials of degree $\le m$ and cannot
be universal; if at least one slice is non-polynomial, the induced
ridge dictionary is dense in $C(X)$ and the corresponding RKHSs are
universal (\prettyref{prop:6}). This dichotomy recovers the usual
division: Gaussian/Laplace/Matérn (universal) versus polynomial kernels
(finite-dimensional).

In \prettyref{sec:3}, we turn to random-kernel networks and Monte
Carlo approximation. With random activations $k(\omega,\cdot,\cdot)$
whose mean is the fixed p.d. $K$, we show: (i) an $L^{2}$ Monte
Carlo bound $\mathbb{E}\|F_{N}-f\|^{2}_{L^{2}(\mu)}\le\|c\|^{2}_{L^{2}(\rho)}/N$
for any $f\in H_{\rho}$ (\prettyref{thm:9}); (ii) a \textit{uniform}
high-probability bound on compact $X$ under mild Lipschitz and compact-support
hypotheses, with dependence on the covering number $\mathcal{N}(\epsilon,X)$
(\prettyref{thm:11}); and (iii) an extension of the uniform result
from continuous-coefficient subclasses to \textit{all} $f\in H_{\rho}$
via a density argument (\prettyref{cor:13}).

These results give a measure-theoretic (and entirely deterministic)
description of the function spaces realized by ridge-kernel averaging,
together with probabilistic finite-sample guarantees for random networks
that need not be p.d. pathwise. Beyond their theoretical interest,
they suggest new architectures using non-p.d. random features, e.g.,
differences of standard kernels, non-monotonic spectral densities,
or oscillatory/structured correlations, where enforcing positive definiteness
\textit{on average} is natural.

\subsection*{Relation to prior work}

Our analysis complements and generalizes the RF program of Rahimi-Recht
\cite{10.5555/2981562.2981710,10.5555/2981780.2981944} and its extensions
(structured RFs \cite{10.5555/3042817.3042964}, optimal kernel-ridge
approximations \cite{MR3662446,pmlr-v70-avron17a}, and operator-theoretic
viewpoints \cite{MR3634888}). From a functional-analytic perspective,
$H_{\rho}$ plays the $L^{2}$ (Hilbert/RKHS) role parallel to Barron's
$L^{1}$ variation spaces for shallow networks \cite{MR1237720},
offering a distinct geometry with inner-product tools. There is also
work on indefinite/signed kernels and signed similarities in learning,
see e.g.,  \cite{MR2042050,10.1145/1015330.1015443,10.1007/978-3-540-45167-9_11};
in contrast to those settings, our randomness acts as a structural
regularizer: the average kernel is p.d. even if samples are not.

\section{Function Classes Induced by Parametric Ridge Kernels}\label{sec:2}

Here we develop the function classes induced by parametric ridge kernels.
Starting from a fixed one-dimensional positive definite kernel, we
compose it with affine projections to form ridge atoms and average
them over a parameter measure. This produces an induced kernel and
an integral function class, which we identify isometrically with its
reproducing kernel Hilbert space. We then reduce to atomic measures
and show that the uniform closure of all induced kernels is exactly
the closure of the conic hull generated by the ridge atoms, yielding
a sharp universality criterion. A slice-wise dichotomy emerges: polynomial
slices give finite-dimensional models, while a single non-polynomial
slice ensures universality on compact sets. We also note a connection
to Barron spaces, contrasting the $L^{1}$ and $L^{2}$ geometry. 

Let $X\subset\mathbb{R}^{d}$ be a compact subset. Fix a continuous
p.d. kernel $K\colon\mathbb{R}\times\mathbb{R}\rightarrow\mathbb{R}$
with $\left|K\left(s,t\right)\right|\leq1$. For each $z=\left(a,b,t\right)\in\mathbb{R}^{d}\times\mathbb{R}\times\mathbb{R}$,
set 
\[
\psi_{z}\left(x\right)\coloneqq K\left(\left\langle a,x\right\rangle +b,t\right),\quad x\in X.
\]
For any finite Borel measure $\rho$ on the parameter space 
\[
Z\coloneqq\mathbb{R}^{d}\times\mathbb{R}\times\mathbb{R}
\]
define
\[
K_{\rho}\left(x,x'\right)\coloneqq\int_{Z}\psi_{z}\left(x\right)\psi_{z}\left(x'\right)d\rho\left(z\right)
\]
on $X\times X$. 
\begin{lem}
\label{lem:1}For every finite Borel measure $\rho$ on $Z$, the
map $K_{\rho}$ is a continuous positive definite kernel on $X$.
Moreover, the space 
\[
H_{\rho}=\left\{ x\mapsto\int_{Z}c\left(z\right)\psi_{z}\left(x\right)d\rho\left(z\right):c\in L^{2}\left(\rho\right)\right\} 
\]
is (isometrically) the RKHS $H_{K_{\rho}}$ with norm 
\[
\left\Vert f\right\Vert _{H_{K_{\rho}}}=\inf\left\{ \left\Vert c\right\Vert _{L^{2}\left(\rho\right)}:f=\int c\psi d\rho\right\} .
\]
\end{lem}

\begin{proof}
The p.d. property of $K_{\rho}$ is immediate, since $K_{\rho}\left(x,x'\right)=\left\langle \psi_{\cdot}\left(x\right)\psi_{\cdot}\left(x'\right)\right\rangle _{L^{2}\left(\rho\right)}$. 

Continuity of $K_{\rho}$ follows from dominated convergence, using
$\left|\psi_{z}\left(x\right)\psi_{z}\left(x'\right)\right|\leq1$
and the continuity of $\left(x,x',z\right)\mapsto\psi_{z}\left(x\right)\psi_{z}\left(x'\right)$
on the compact set $X\times X\times supp\left(\rho|_{S}\right)$ after
truncating $\rho$ to a large compact $S\subset Z$ and controlling
the tail by its mass. 

The standard construction of RKHSs from feature maps gives the identification
of $H_{K_{\rho}}$ and the norm formula. See e.g., \cite{MR51437}
and \cite[Theorem 4.21]{MR2450103}.
\end{proof}
To understand how flexible the family $\left\{ K_{\rho}\right\} $
is as $\rho$ varies, we first reduce to the case of atomic measures. 
\begin{lem}[Approximation by finite conic sums]
\label{lem:2} Let $X\subset\mathbb{R}^{d}$ be compact. For every
finite measure $\rho$ on $Z$ and $\epsilon>0$, there exist $z_{1},\dots,z_{m}\in Z$
and nonnegative weights $w_{1},\dots,w_{m}$ such that 
\[
\sup_{x,x'\in X}\left|K_{\rho}\left(x,x'\right)-\sum^{m}_{j=1}w_{j}\psi_{z_{j}}\left(x\right)\psi_{z_{j}}\left(x'\right)\right|<\epsilon.
\]
Thus every $K_{\rho}$ is a uniform limit (on compacts) of conic combinations
of the rank-one kernels $k_{z}\left(x,x'\right)\coloneqq\psi_{z}\left(x\right)\psi_{z}\left(x'\right)$. 
\end{lem}

\begin{proof}
Truncate $\rho$ to a compact set so the contribution of the tail
is small. On the compact domain, the map $\left(x,x',z\right)\mapsto\psi_{z}\left(x\right)\psi_{z}\left(x'\right)$
is uniformly continuous. Take a finite partition $\left\{ A_{j}\right\} $
of the support of $\rho$, select representatives $z_{j}$ and approximate
the integral by a finite sum with weights $w_{j}=\rho\left(A_{j}\right)$.
Uniform approximation follows. 
\end{proof}
Let 
\[
\mathcal{F}\coloneqq\overline{\left\{ \psi_{z}:z\in Z\right\} }^{\left\Vert \cdot\right\Vert _{\infty}}\subset C\left(X\right)
\]
and define the cone 
\[
Cone\left(\mathcal{F}\right)\coloneqq\left\{ \sum^{m}_{j=1}w_{j}f_{j}\left(x\right)f_{j}\left(x'\right):w_{j}\geq0,\:f_{j}\in\mathcal{F}\right\} .
\]

\begin{prop}
\label{prop:3}For every compact $X\subset\mathbb{R}^{d}$, the uniform
closure of the family $\left\{ K_{\rho}:\rho\;\text{finite}\right\} $
coincides with the uniform closure of $Cone\left(\mathcal{F}\right)$,
that is, 
\[
\overline{\left\{ K_{\rho}:\rho\:\text{finite }\right\} }^{\left\Vert \cdot\right\Vert _{\infty}}=\overline{Cone\left(\mathcal{F}\right)}^{\left\Vert \cdot\right\Vert _{\infty}}.
\]
\end{prop}

\begin{proof}
\prettyref{lem:2} shows that any $K_{\rho}$ can be uniformly approximated
by finite conic sums with atoms in $\mathcal{F}$. Conversely, any
finite conic sum with $f_{j}\in\mathcal{F}$ can be approximated by
conic sums with atoms $\psi_{z_{j}}$, hence by an atomic $\rho$.
Uniform limits preserve positivity and continuity, hence the closures
agree. 
\end{proof}
Thus, varying $\rho$ does not allow arbitrary kernels, but only those
that lie in the conic hull generated by the set of ridge atoms. More
precisely:
\begin{cor}
\label{cor:4}This model can generate all continuous p.d. kernels
on $X$ if and only if $\mathcal{F}=C\left(X\right)$, i.e., if the
atom set itself is sup-norm dense in $C\left(X\right)$. 
\end{cor}

\begin{proof}
Every continuous p.d. kernel $L$ on $X$ admits a Mercer expansion
\[
L\left(x,x'\right)=\sum^{\infty}_{n=1}\lambda_{n}\varphi_{n}\left(x\right)\varphi_{n}\left(x'\right)
\]
with $\lambda_{n}\geq0$, $\varphi_{n}\in C\left(X\right)$, converging
uniformly on $X\times X$. Hence the cone generated by all continuous
functions $\left\{ g\otimes g:g\in C\left(X\right)\right\} $ is dense
in the space $PD\left(X\right)$ of all continuous p.d. kernels on
$X$. 

By \prettyref{prop:3}, our construction generates instead the cone
$Cone\left(\mathcal{F}\right)$. Therefore we obtain all of $PD\left(X\right)$
if and only if $\mathcal{F}=C\left(X\right)$.
\end{proof}
\begin{rem}
The requirement $\mathcal{F}=C\left(X\right)$ is very strong. In
practice, most ridge dictionaries have the property that the linear
span of $\left\{ \psi_{z}\right\} $ is dense (the classical universal
approximation property), but not that the set itself is dense. Consequently,
for typical kernels $K$ (e.g. Gaussian), the induced $H_{\rho}$
are rich and yield universal function approximation, but the family
$\left\{ K_{\rho}\right\} $ does not exhaust all $PD\left(X\right)$. 
\end{rem}

The extent to which the family $\left\{ H_{\rho}:\rho\:\text{finite}\right\} $
can approximate arbitrary functions on a compact domain depends critically
on the analytic structure of the one-dimensional slices $s\mapsto K\left(s,t\right)$,
$t\in\mathbb{R}$. 
\begin{prop}
\label{prop:6}Let $X\subset\mathbb{R}^{d}$ be compact with nonempty
interior. Then:
\begin{enumerate}
\item If for every $t$ the function $s\mapsto K\left(s,t\right)$ is a
polynomial of degree at most $m$, then every atom $\psi_{z}\left(x\right)=K\left(\left\langle a,x\right\rangle +b,t\right)$
is a polynomial in $\left\langle a,x\right\rangle +b$ of degree at
most $m$. Consequently, for all finite $\rho$, the space $H_{\rho}$
is contained in the space of polynomials on $\mathbb{R}^{d}$ of degree
at most $m$. In particular, the family $\left\{ H_{\rho}\right\} $
can not be universal. 
\item If there exists $t_{0}$ such that $s\mapsto K\left(s,t\right)$ is
not a polynomial, then the linear span of the dictionary 
\[
G\coloneqq\left\{ x\mapsto K\left(\left\langle a,x\right\rangle +b,t_{0}\right):\left(a,b\right)\in\mathbb{R}^{d}\times\mathbb{R}\right\} 
\]
is dense in $C\left(X\right)$ under the uniform norm. Hence for any
finite measure $\rho$ that assigns positive mass to the slice $\left\{ t=t_{0}\right\} $,
the corresponding kernel $K_{\rho}$ induces an RKHS that is universal
on $X$. 
\end{enumerate}
\end{prop}

\begin{proof}
First, if $s\mapsto K\left(s,t\right)$ is polynomial of degree at
most $m$, then $\psi_{z}\left(x\right)$ is a degree $m$ polynomial
in the affine form $\left\langle a,x\right\rangle +b$. Expanding
gives a polynomial in the coordinates of $x$ of degree at most $m$.
Thus every $f\in H_{\rho}$ is contained in the finite-dimensional
space of multivariate polynomials of degree $\leq m$. 

On the other hand, if $s\mapsto K\left(s,t_{0}\right)$ is non-polynomial
and continuous, classical results on ridge-function approximation
theorems (\cite{MR1015670,MR1819645,HORNIK1991251}) show that the
span of $G$ is dense in $C\left(X\right)$. Hence $G$ is a universal
dictionary. If $\rho$ assigns positive mass to the slice $\left\{ t=t_{0}\right\} $,
then $H_{\rho}$ contains linear combinations of atoms from $G$,
so the RKHS is universal. 
\end{proof}
This yields a sharp dichotomy:

(1) Non-polynomial slices (e.g. Gaussian, Laplace, Cauchy, Matérn,
and most kernels of interest): the induced spaces $H_{\rho}$ can
achieve universal approximation.

(2) Polynomial slices (e.g. polynomial kernels, truncated expansions):
the induced spaces are finite-dimensional and thus not universal.
\begin{rem}[Connection with Barron spaces]
 The function spaces $H_{\rho}$ introduced above are closely related
to the Barron spaces that appear in the approximation theory of shallow
neural networks. In the seminal paper \cite{MR1237720}, Barron considered
functions on $\mathbb{R}^{d}$ that can be represented as
\[
f\left(x\right)=\int_{\mathbb{R}^{d}\times\mathbb{R}}a\sigma\left(\left\langle w,x\right\rangle +b\right)d\pi\left(w,b,a\right),\quad\int\left|a\right|d\pi<\infty,
\]
where $\sigma\colon\mathbb{R}\rightarrow\mathbb{R}$ is a fixed activation
function. The Baron norm of $f$ is the minimal variation $\int\left|a\right|d\pi$
among all such representations. This choice of $L^{1}$ type constraint
on the coefficient measure yields a Banach space whose geometry is
well suited to control the expressive power of shallow neural networks.
Indeed, functions with small Barron norm admit approximation by two-layer
networks with a number of neurons proportional to the inverse of the
approximation accuracy (see \cite{MR1237720})

Our spaces 
\[
H_{\rho}=\left\{ f\left(x\right)=\int_{Z}c\left(z\right)\psi_{z}\left(x\right)d\rho\left(z\right):c\in L^{2}\left(\rho\right)\right\} 
\]
where $\psi_{z}\left(x\right)=K\left(\left\langle a,x\right\rangle +b,t\right)$,
$z=\left(a,b,t\right)$, use the same ridge-function dictionary but
impose an $L^{2}$ constraint on the coefficient function $c$. This
yields a Hilbert space structure: the norm is given by 
\[
\left\Vert f\right\Vert _{H_{\rho}}=\inf\left\{ \left\Vert c\right\Vert _{L^{2}\left(\rho\right)}:f\left(x\right)=\int_{Z}c\left(z\right)\psi_{z}\left(x\right)d\rho\left(z\right)\right\} 
\]
and $H_{\rho}$ is isometrically identified with the RKHS $H_{K_{\rho}}$
of the kernel $K_{\rho}\left(x,x'\right)=\int_{Z}\psi_{z}\left(x\right)\psi_{z}\left(x'\right)d\rho\left(z\right)$. 

Thus, Barron spaces and the spaces $H_{\rho}$ are two parallel ways
of organizing ridge-function expansions:

(1) Barron spaces: $L^{1}$-geometry on the coefficient space, leading
to Banach spaces with norms tied to variation and convex approximation.
These are central in generalization analysis of shallow networks.

(2) $L^{2}$-geometry on the coefficient space, leading to Hilbert
spaces and the well-developed machinery of RKHS theory and kernel
methods.
\end{rem}

In this sense, $H_{\rho}$ may be seen as a Hilbertian analogue of
the Barron space. Both constructions describe functions representable
by ridge expansions, but differ in their choice of coefficient norm
and hence in their induced geometry: the $L^{1}$ setting emphasizes
sparsity and convexity, and the $L^{2}$ setting yields inner product
structure and spectral tools. This connection suggests that many approximation
theoretic results in the Barron setting (see, e.g., \cite{MR1237720,MR3634886,MR4375792})
may have analogues in the RKHS $H_{\rho}$.

\section{Random-Kernel Networks}\label{sec:3}

The function classes introduced in the previous sections arise naturally
in the study of random feedforward networks. We now show how the ridge-function
spaces $H_{\rho}$ and their RKHS representations appear as the limiting
function spaces of single hidden-layer networks with random kernel
activations.

\subsection*{Architecture and mean kernel}

Let $(\Omega,\mathcal{F},\mathbb{P})$ be a probability space. Consider
a jointly measurable function 
\[
k\colon\Omega\times\mathbb{R}\times\mathbb{R}\to\mathbb{R},\qquad\left(\omega,s,t\right)\mapsto k\left(\omega,s,t\right)
\]
which we interpret as a random kernel-valued function. For each fixed
pair $\left(s,t\right)\in\mathbb{R}^{2}$, the map $k\left(\cdot,s,t\right)$
is a real-valued random variable on $\Omega$.

Define the mean kernel
\[
K\left(s,t\right)=\mathbb{E}_{\omega}\left[k\left(\cdot,s,t\right)\right].
\]
Assume throughout that $K$ is continuous and positive definite, but
we do \textbf{not }assume $k\left(\cdot,s,t\right)$ is p.d. pathwise.
In particular, the random features may be signed or indefinite.

Let $Z:=\mathbb{R}^{d}\times\mathbb{R}\times\mathbb{R}$ denote the
parameter space for affine ridge maps, with generic element $z=\left(a,b,t\right)$.
Let $\rho$ be a Borel probability measure on $Z$. As in \prettyref{sec:2},
we regard $z$ as specifying an affine projection $\mathbb{R}^{d}\ni x\mapsto\left\langle a,x\right\rangle +b$,
followed by evaluation at scale/location parameter $t$. 

We construct the network using two independent sequences of i.i.d.
samples: $z_{i}=(a_{i},b_{i},t_{i})\sim\rho$ and kernel states $\omega_{i}\sim\mathbb{P}$,
independent across indices. The activation of the $i$-th random neuron
is defined as
\[
\phi_{i}(x):=k(\omega_{i},\langle a_{i},x\rangle+b_{i},t_{i}),\qquad x\in\mathbb{R}^{d}.
\]

For $N\ge1$ and output weights $\alpha_{1},\dots,\alpha_{N}\in\mathbb{R}$,
the single hidden-layer random network is the linear combination
\begin{equation}
F_{N}\left(x\right)\coloneqq\sum^{N}_{i=1}\alpha_{i}\phi_{i}\left(x\right)=\sum^{N}_{i=1}\alpha_{i}k\left(\omega_{i},\left\langle a_{i},x\right\rangle +b_{i},t_{i}\right),\quad x\in\mathbb{R}^{d}.\label{eq:b1}
\end{equation}

The training task reduces to determining the weights $\alpha_{i}$.
In the standard supervised learning setup, given a dataset $\left\{ \left(x_{j},y_{j}\right)\right\} ^{M}_{j=1}$
in $\mathbb{R}^{d}\times\mathbb{R}$, the coefficients are obtained
by solving a linear regression problem
\[
\min_{\alpha\in\mathbb{R}^{N}}\frac{1}{M}\sum^{M}_{j=1}\left|y_{i}-F_{N}\left(x\right)\right|^{2}+\lambda\left\Vert \alpha\right\Vert ^{2}_{2},
\]
where $\lambda\geq0$ is a regularization parameter. In matrix form,
letting $\Phi\in\mathbb{R}^{M\times N}$ be the feature matrix with
entries $\Phi_{ji}=\phi_{i}\left(x_{j}\right)$, the optimal weights
are given by the ridge regression solution
\[
\hat{\alpha}=\mathop{argmin}_{\alpha}\left(\left\Vert y-\Phi\alpha\right\Vert ^{2}_{2}+\lambda\left\Vert \alpha\right\Vert ^{2}_{2}\right)=\left(\Phi^{*}\Phi+\lambda I\right)^{-1}\Phi^{*}y.
\]
Hence, training the network consists only of solving a convex quadratic
problem in $\mathbb{R}^{N}$. The nonlinear structure of the features
comes entirely from the random sampling procedure; no backpropagation
through the random kernel activations is required.

\subsection*{Lifted kernel and limiting function space}

Define the lifted kernel on $\mathbb{R}^{d}\times\mathbb{R}^{d}$
\begin{align*}
K_{\rho}\left(x,x'\right) & \coloneqq\mathbb{E}_{z\sim\rho}\left[\psi_{z}\left(x\right)\psi_{z}\left(x'\right)\right]\\
 & =\int_{Z}K\left(\left\langle a,x\right\rangle +b,t\right)K\left(\left\langle a,x'\right\rangle +b,t\right)d\rho\left(a,b,t\right)
\end{align*}
which is positive definite by construction. As shown in \prettyref{lem:1},
the limiting function class of random-kernel networks coincides with
the Hilbert spaces $H_{\rho}$ described in \prettyref{sec:2}.

\subsection*{Approximation by finite networks}

The finite network $F_{N}$ can be viewed as a Monte Carlo approximation
to such $f\in H_{\rho}$. The following theorems shows that this approximation
holds in both $L^{2}$ (\prettyref{thm:9}) and uniform senses (\prettyref{thm:11},
\prettyref{cor:13}), under mild assumptions. 

For notational clarity, we consider the canonical unbiased choice
of output weights: 
\[
F_{N}\left(x\right)=\frac{1}{N}\sum^{N}_{i=1}c\left(z_{i}\right)k\left(\omega_{i},\left\langle a_{i},x\right\rangle +b,t_{i}\right),
\]
where $z_{i}=\left(a_{i},b_{i},t_{i}\right)$ and $c\in L^{2}\left(\rho\right)$
is the coefficient function appearing in the representation of $f$
below. 
\begin{rem}
The coefficient function $c$ is in $L^{2}\left(\rho\right)$, which
is formally an equivalence class of measurable functions defined up
to $\rho$-null sets. When we write $c\left(z_{i}\right)$ for i.i.d.
samples $z_{i}\sim\rho$, what is meant is the following: fix any
measurable representative $c^{*}$ of the equivalence class $c$,
and define the network using $c^{*}\left(z_{i}\right)$. Since the
probability that a sample $z_{i}$ falls into a $\rho$-null set is
zero, the resulting random network is almost surely independent of
the choice of representative. Thus all formula involving $c\left(z_{i}\right)$
are well defined almost surely and depend only on the $L^{2}\left(\rho\right)$-class
of $c$. This convention is standard in Monte Carlo approximation. 
\end{rem}

\begin{thm}[$L^{2}$ approximation]
\label{thm:9} Let $f\in\mathcal{H}_{\rho}$ have the representation
\[
f\left(x\right)=\int_{Z}c\left(a,b,t\right)K\left(\left\langle a,x\right\rangle +b,t\right)d\rho\left(a,b,t\right)
\]
for some $c\in L^{2}(\rho)$ with $\|c\|_{L^{2}(\rho)}\le C$. Let
$\left\{ \left(a_{i},b_{i},t_{i}\right)\right\} {}^{N}_{i=1}\sim\rho$
and $\left\{ \omega_{i}\right\} ^{N}_{i=1}\sim\mathbb{P}$ be i.i.d.
mutually independent. Assume $\left|k\left(\omega,s,t\right)\right|\leq1$
almost surely. Then, for any probability measure $\mu$ on $\mathbb{R}^{d}$,
\[
\mathbb{E}\left\Vert F_{N}-f\right\Vert ^{2}_{L^{2}(\mu)}\le\frac{C^{2}}{N},\qquad\mathbb{P}\left(\left\Vert F_{N}-f\right\Vert _{L^{2}(\mu)}>\epsilon\right)\le\frac{C^{2}}{N\epsilon^{2}}.
\]
\end{thm}

\begin{proof}
For each $i$, define 
\[
Y_{i}\left(x\right)\coloneqq c\left(z_{i}\right)k\left(\omega_{i},\left\langle a_{i},x\right\rangle +b_{i},t_{i}\right).
\]
Since $\left|k\right|\leq1$ almost surely, for each realization of
$\left(\omega_{i},z_{i}\right)$, 
\[
\int\left|Y_{i}\left(x\right)\right|^{2}d\mu\left(x\right)=\left|c\left(z_{i}\right)\right|^{2}\int\left|k\left(\omega_{i},\left\langle a_{i},x\right\rangle +b,t_{i}\right)\right|^{2}d\mu\left(x\right)\leq\left|c\left(z_{i}\right)\right|^{2}<\infty.
\]
Thus $Y_{i}\left(\cdot\right)\in L^{2}\left(\mu\right)$ almost surely,
i.e., $Y_{i}$ is an $L^{2}\left(\mu\right)$-valued random variable.
(See, e.g., \cite{MR982264,MR1102015} for the general theory of of
random variables in Banach spaces.)

The Monte Carlo network can then be written as 
\[
F_{N}(x)=\frac{1}{N}\sum^{N}_{i=1}Y_{i}(x)
\]
Unbiasedness holds pointwise: For each $x$, 
\begin{align*}
\mathbb{\mathbb{E}}\left[Y_{i}\left(x\right)\mid z_{i}\right] & =c\left(z_{i}\right)\mathbb{E}_{\omega}\left[k\left(\omega,\left\langle a_{i},x\right\rangle +b_{i},t_{i}\right)\right]\\
 & =c\left(z_{i}\right)K\left(\left\langle a_{i},x\right\rangle +b_{i},t_{i}\right),
\end{align*}
Taking expectation over $z_{i}\sim\rho$ gives $\mathbb{E}\left[Y_{i}\left(x\right)\right]=f\left(x\right)$.
Equivalently, $\mathbb{E}\left[Y_{i}\right]=f\in L^{2}\left(\mu\right)$.
Hence also $\mathbb{E}\left[F_{N}\right]=f$. 

Since $Y_{i}$ are i.i.d. in the Hilbert space $L^{2}\left(\mu\right)$
with mean $f$, 
\begin{align*}
\mathbb{E}\left\Vert F_{N}-f\right\Vert ^{2}_{L^{2}\left(\mu\right)} & =\frac{1}{N}\mathbb{E}\left\Vert Y_{1}-f\right\Vert ^{2}_{L^{2}\left(\mu\right)}\\
 & =\frac{1}{N}\left(\mathbb{E}\left\Vert Y_{1}\right\Vert ^{2}_{L^{2}\left(\mu\right)}-\left\Vert f\right\Vert ^{2}_{L^{2}\left(\mu\right)}\right)\leq\frac{1}{N}\left(\mathbb{E}\left\Vert Y_{1}\right\Vert ^{2}_{L^{2}\left(\mu\right)}\right).
\end{align*}
Using $\left|k\left(\omega,s,t\right)\right|\leq1$ a.s. and Fubini,
\begin{align*}
\mathbb{E}\left[\left\Vert Y_{1}\right\Vert ^{2}_{L^{2}\left(\mu\right)}\right] & =\mathbb{E}_{\omega,z}\left[\int\left|c\left(z\right)\right|^{2}\left|k\left(\omega,\left\langle a,x\right\rangle +b,t\right)\right|^{2}d\mu\left(x\right)\right]\\
 & \leq\mathbb{E}_{z}\left[\left|c\left(z\right)\right|^{2}\right]=\left\Vert c\right\Vert ^{2}_{L^{2}\left(\rho\right)}\leq C^{2}.
\end{align*}
This gives 
\[
\mathbb{E}\left\Vert F_{N}-f\right\Vert ^{2}_{L^{2}(\mu)}\le C^{2}/N.
\]
Finally, Chebyshev's inequality yields 
\[
\mathbb{P}\left(\left\Vert F_{N}-f\right\Vert _{L^{2}(\mu)}>\epsilon\right)\le C^{2}/\left(N\epsilon^{2}\right).
\]
 
\end{proof}
The bound in \prettyref{thm:9} controls the mean-square error with
respect to an arbitrary sampling measure on $\mathbb{R}^{d}$. In
many applications one needs a uniform guarantee on a compact set $X\subset\mathbb{R}^{d}$.
A sup--norm bound requires mild regularity/ boundedness assumptions
to pass from finitely many control points (an $\epsilon$-net of $X$)
to all of $X$. We now state and prove this uniform result.
\begin{lem}
\label{lem:10}Let $\rho$ be a Borel probability measure on $Z:=\mathbb{R}^{d}\times\mathbb{R}\times\mathbb{R}$
with compact support $S:=supp\left(\rho\right)$. Assume $\left|k\left(\omega,s,t\right)\right|\le1$
almost surely and define $K\left(s,t\right)\coloneqq\mathbb{E}_{\omega}\left[k\left(\omega,s,t\right)\right]$.
For $c\in L^{2}(\rho)$ set
\begin{equation}
f_{c}\left(x\right)\coloneqq\int_{S}c\left(a,b,t\right)K\left(\left\langle a,x\right\rangle +b,t\right)d\rho\left(a,b,t\right).\label{eq:c2}
\end{equation}
Then the subclass
\begin{equation}
H^{\mathrm{cont}}_{\rho}\coloneqq\left\{ f_{c}:c\in C(S)\right\} \label{eq:c3}
\end{equation}
is dense in $H_{\rho}=\left\{ f_{c}:c\in L^{2}\left(\rho\right)\right\} $
for the uniform norm on any compact $X\subset\mathbb{R}^{d}$.
\end{lem}

\begin{proof}
Since $S$ is compact and $\rho$ is a finite Borel measure, $C\left(S\right)$
is dense in $L^{2}\left(\text{\ensuremath{\rho}}\right)$. Also $\left|K\left(s,t\right)\right|\le\mathbb{E}_{\omega}\left[\left|k\left(\cdot,s,t\right)\right|\right]\le1$
for all $\left(s,t\right)$. 

Given $c\in L^{2}\left(\rho\right)$, pick $c_{n}\in C\left(S\right)$
with $c_{n}\to c$ in $L^{2}\left(\rho\right)$. For any $x\in X$,

\begin{align*}
\left|f_{c_{n}}\left(x\right)-f_{c}\left(x\right)\right| & =\left|\int_{S}\left(c_{n}-c\right)\left(z\right)K\left(\left\langle a,x\right\rangle +b,t\right)d\rho\left(z\right)\right|\\
 & \le\int_{S}\left|c_{n}-c\right|\left|K\right|d\rho\le\|c_{n}-c\|_{L^{2}\left(\rho\right)}
\end{align*}
using the fact that $\rho$ is a probability measure. Then, 
\[
\sup_{x\in X}\left|f_{c_{n}}\left(x\right)-f_{c}\left(x\right)\right|\le\|c_{n}-c\|_{L^{2}\left(\rho\right)}\xrightarrow[n\to\infty]{}0
\]
and so $H^{cont}_{\rho}$ is dense in $H_{\rho}$ under the $\|\cdot\|_{\infty}$-norm
on $X$. 
\end{proof}
\begin{thm}[Uniform approximation in $H^{cont}_{\rho}$ on compact sets]
\label{thm:11} Let $X\subset\mathbb{R}^{d}$ be compact. Assume:
\begin{itemize}
\item $k(\omega,s,t)$ is $L_{k}$-Lipschitz in $s$, uniformly in $t$,
almost surely;
\item $\rho$ is supported on the product ball 
\begin{equation}
S=\left\{ \left(a,b,t\right)\in Z=\mathbb{R}^{d}\times\mathbb{R}\times\mathbb{R}:\left\Vert a\right\Vert _{2}\leq1,\:\left|b\right|\leq1,\:\left|t\right|\leq1\right\} ;\label{eq:c-4}
\end{equation}
\item $\left|k\left(\omega,s,t\right)\right|\le1$ almost surely.
\end{itemize}
Let $f=f_{c}\in H^{cont}_{\rho}$ (see \eqref{eq:c2}-\eqref{eq:c3}),
and set $C\coloneqq\left\Vert c\right\Vert _{L^{\infty}\left(S\right)}$.
Draw $z_{i}=\left(a_{i},b_{i},t_{i}\right)\stackrel{\mathrm{i.i.d.}}{\sim}\rho$
and $\omega_{i}\stackrel{\mathrm{i.i.d.}}{\sim}\mathbb{P}$ independently,
and define
\[
F_{N}(x):=\frac{1}{N}\sum^{N}_{i=1}c(z_{i})\,k\big(\omega_{i},\langle a_{i},x\rangle+b_{i},t_{i}\big).
\]
Then, for any $\delta\in(0,1)$ and $\epsilon>0$, with probability
at least $1-\delta$ (over the joint draw of $\left\{ z_{i},\omega_{i}\right\} ^{N}_{i=1}$),
\begin{equation}
\sup_{x\in X}\left|F_{N}\left(x\right)-f\left(x\right)\right|\le C\sqrt{\frac{2\log\left(2\mathcal{N}\left(\epsilon,X\right)\right)+2\log\left(1/\delta\right)}{N}}+2CL_{k}\epsilon,\label{eq:c-5}
\end{equation}
where $\mathcal{N}\left(\epsilon,X\right)$ denotes the $\epsilon$-covering
number of $X$ in the Euclidean metric.

\end{thm}

\begin{proof}
Fix $\epsilon>0$ and let $\left\{ x^{1},\dots,x^{M}\right\} $ be
an $\epsilon$-net of $X$, with $M=\mathcal{N}\left(\epsilon,X\right)<\infty$.
For each net point $x^{m}$, define 
\[
Y_{i}\left(x^{m}\right)\coloneqq c\left(z_{i}\right)k\left(\omega_{i},\left\langle a_{i},x^{m}\right\rangle +b_{i},t_{i}\right),\quad F_{N}\left(x^{m}\right)=\frac{1}{N}\sum^{N}_{i=1}Y_{i}\left(x^{m}\right).
\]
Since $\left|c\left(z_{i}\right)\right|\leq C$ and $\left|k\right|\leq1$
almost surely, $\left|Y_{i}\left(x^{m}\right)\right|\leq C$ almost
surely. Moreover, 
\begin{align*}
\mathbb{E}\left[Y_{i}\left(x^{m}\right)\mid z_{i}\right] & =c\left(z_{i}\right)\mathbb{E}_{\omega}\left[k\left(\omega_{i},\left\langle a_{i},x^{m}\right\rangle +b_{i},t_{i}\right)\right]\\
 & =c\left(z_{i}\right)K\left(\left\langle a_{i},x^{m}\right\rangle +b_{i},t_{i}\right),
\end{align*}
hence, averaging over $z_{i}\sim\rho$, 
\[
\mathbb{E}\left[Y_{i}\left(x^{m}\right)\right]=f\left(x^{m}\right).
\]
By Hoeffding's inequality, 
\[
\mathbb{P}\left(\left|F_{N}\left(x^{m}\right)-f\left(x^{m}\right)\right|>\theta\right)\leq2\exp\left(-\frac{N\theta^{2}}{2C^{2}}\right),\quad\theta>0.
\]
A union bound over the $M$ net points gives 
\begin{align*}
\mathbb{P}\left(\max_{1\leq m\leq M}\left|F_{N}\left(x^{m}\right)-f\left(x^{m}\right)\right|>\theta\right) & =\mathbb{P}\left(\bigcup^{M}_{m=1}\left\{ \left|F_{N}\left(x^{m}\right)-f\left(x^{m}\right)\right|>\theta\right\} \right)\\
 & \leq2M\exp\left(-\frac{N\theta^{2}}{2C^{2}}\right),\quad\theta>0.
\end{align*}
Setting the right-hand side to $\delta$ and solving for $\theta$,
we get that, with probability at least $1-\delta$, 
\begin{align}
\max_{1\leq m\leq M}\left|F_{N}\left(x^{m}\right)-f\left(x^{m}\right)\right| & \leq C\sqrt{\frac{2\log\left(2\mathcal{N}\left(\epsilon,X\right)\right)+2\log\left(1/\delta\right)}{N}}.\label{eq:c5}
\end{align}

By the Lipschitz assumption on $k$, for all $s,s',t\in\mathbb{R}$,
\[
\left|k\left(\omega,s,t\right)-k\left(\omega,s',t\right)\right|\leq L_{k}\left|s-s'\right|,\quad a.s.
\]
It follows that, for any $x,x'\in\mathbb{R}^{d}$, 
\begin{align*}
\left|k\left(\omega,\left\langle a,x\right\rangle +b,t\right)-k\left(\omega,\left\langle a,x'\right\rangle +b,t\right)\right| & \leq L_{k}\left|\left\langle a,x-x'\right\rangle \right|\\
 & \leq L_{k}\left\Vert a\right\Vert \left\Vert x-x'\right\Vert ,\quad a.s.
\end{align*}
Since $\rho$ is supported on the product ball $S$ (see \eqref{eq:c-4}),
we have $\left\Vert a\right\Vert \leq1$, and so the map $x\mapsto k\left(\omega,\left\langle a,x\right\rangle +b,t\right)$
is $L_{k}$-Lipschitz, hence 
\[
x\mapsto c\left(z_{i}\right)k\left(\omega,\left\langle a,x\right\rangle +b,t\right),\quad a.s.
\]
is $CL_{k}$-Lipschitz, and so 
\[
Lip\left(F_{N}\right)\leq CL_{k},\quad a.s.
\]

Note that $K\left(s,t\right)=\mathbb{E}_{\omega}\left[k\left(\omega,s,t\right)\right]$
is also $L_{k}$-Lipschitz in $s$ uniformly in $t$, since
\[
\left|K\left(s,t\right)-K\left(s',t\right)\right|\leq\mathbb{E}_{\omega}\left[\left|k\left(\omega,s,t\right)-k\left(\omega,s',t\right)\right|\right]\leq L_{k}\left|s-s'\right|.
\]
This means each atom $x\mapsto K\left(\left\langle a,x\right\rangle +b,t\right)$
is $L_{k}$-Lipschitz ($\left\Vert a\right\Vert \leq1$ by assumption),
and so 
\[
Lip\left(f\right)\leq CL_{k}.
\]

Now for any $x\in X$, let $x^{m}$ be the net point with $\left\Vert x-x^{m}\right\Vert <\epsilon$.
Then, 
\begin{align*}
\left|F_{N}\left(x\right)-f\left(x\right)\right| & \leq\left|F_{N}\left(x\right)-F_{N}\left(x^{m}\right)\right|+\left|F_{N}\left(x^{m}\right)-f\left(x^{m}\right)\right|+\left|f\left(x^{m}\right)-f\left(x\right)\right|\\
 & \leq\left(Lip\left(F_{N}\right)+Lip\left(f\right)\right)\epsilon+\left|F_{N}\left(x^{m}\right)-f\left(x^{m}\right)\right|\\
 & =2CL_{k}\epsilon+\left|F_{N}\left(x^{m}\right)-f\left(x^{m}\right)\right|,
\end{align*}
so that 
\begin{equation}
\sup_{x\in X}\left|F_{N}\left(x\right)-f\left(x\right)\right|\leq2CL_{k}\epsilon+\max_{1\leq m\leq M}\left|F_{N}\left(x^{m}\right)-f\left(x^{m}\right)\right|.\label{eq:c7}
\end{equation}
Combining \eqref{eq:c5} with \eqref{eq:c7} gives that stated inequality
\eqref{eq:c-5}.
\end{proof}
\begin{cor}
If $X\subset\mathbb{R}^{d}$ is contained in a Euclidean ball of radius
$D$, then with probability at least $1-\delta$, 
\begin{equation}
\sup_{x\in X}\left|F_{N}\left(x\right)-f\left(x\right)\right|=\tilde{\mathcal{O}}\left(C\sqrt{\frac{d\log N+\log1/\delta}{N}}\right).\label{eq:c8}
\end{equation}
\end{cor}

\begin{proof}
By a standard volumetric argument (see, e.g., \cite[Corollary 4.2.13]{MR3837109}),
if $X\subset\mathbb{R}^{d}$ is contained in a ball of radius $D$,
then $\mathcal{N}\left(\epsilon,X\right)\leq\left(sD/\epsilon\right)^{d}$
for a universal constant $s>0$. Substituting into \prettyref{eq:c-5}
gives
\begin{align}
\sup_{x\in X}\left|F_{N}\left(x\right)-f\left(x\right)\right| & \approx C\sqrt{\frac{d\log\left(sD/\epsilon\right)}{N}}+2CL_{k}\epsilon\label{eq:3.9}
\end{align}
Minimizing over $\epsilon$: Setting $f\left(\epsilon\right)=C\sqrt{\frac{d\log\left(sD/\epsilon\right)}{N}}+2CL_{k}\epsilon$,
then 
\[
f'\left(\epsilon\right)=-\frac{C}{2}\sqrt{\frac{d}{N}}\frac{1}{\epsilon\sqrt{\log\left(sD/\epsilon\right)}}+2CL_{k}=0
\]
implies that 
\[
\epsilon=\frac{1}{4L_{k}}\sqrt{\frac{d}{N}}\frac{1}{\sqrt{\log\left(sD/\epsilon\right)}}.
\]
At the optimum $\log\left(sD/\epsilon\right)$ is of order $\log N$,
hence a near optimal choice is 
\[
\epsilon^{*}\approx\frac{1}{L_{k}}\sqrt{\frac{d}{N\log N}}.
\]
Substituting $\epsilon^{*}$ into \eqref{eq:3.9} gives the claimed
bound from \eqref{eq:c8}.
\end{proof}
\begin{cor}[uniform approximation in $H_{\rho}$ on compact sets]
\label{cor:13}Assume the hypotheses of \prettyref{lem:10} and \prettyref{thm:11},
and let $X\subset\mathbb{R}^{d}$ be compact. For any $f\in H_{\rho}$,
$\eta>0$ and $\delta\in\left(0,1\right)$, there exists a coefficient
$c^{*}\in C\left(S\right)$, $\epsilon>0$, and $N\in\mathbb{N}$
such that, if 
\[
F_{N}\left(x\right)=\frac{1}{N}\sum^{N}_{i=1}c^{*}\left(z_{i}\right)k\left(\omega_{i},\left\langle a_{i},x\right\rangle +b_{i},t_{i}\right)
\]
is constructed as in \prettyref{thm:11} using $c^{*}$, then with
probability at least $1-\delta$ (over the joint draw of $\left\{ z_{i},\omega_{i}\right\} ^{N}_{i=1}$),
\[
\sup_{x\in X}\left|F_{N}\left(x\right)-f\left(x\right)\right|<\eta.
\]
\end{cor}

\begin{proof}
(sketch) By \prettyref{lem:10}, the subclass $H^{cont}_{\rho}$ is
uniformly dense in $H_{\rho}$ on any compact $X$. Given $f\in H_{\rho}$
and $\eta>0$, choose $c^{*}\in C\left(S\right)$ so that 
\[
\sup_{x\in X}\left|f_{c^{*}}\left(x\right)-f\left(x\right)\right|\leq\eta/3.
\]
Apply \prettyref{thm:11} to $f_{c^{*}}$: choose $\epsilon$ and
$N$ so that its high probability bound is $\leq2\eta/3$. By the
triangle inequality, $\sup_{x}\left|F_{N}\left(x\right)-f\left(x\right)\right|<\eta$
with probability $\geq1-\delta$.
\end{proof}
\bibliographystyle{plain}
\bibliography{ref}

\end{document}